\documentclass[a4paper,11pt]{amsart}
\newcommand{\R}{\mathbb{R}}
\newcommand{\Z}{\mathbb{Z}}
\newcommand{\rat}{R}
\newcommand{\s}{\mathcal{S}}
\def\harf{
\ifinner
    {\frac{1}{2}}
\else
    {\hbox{$\frac{1}{2}$}}
\fi
}

\usepackage{amssymb,epsfig,amsmath,amsthm}

\begin{document}

\title[Markov chain approximations to diffusions]{The non-locality of Markov chain 
approximations to two-dimensional 
diffusions}

\author{C. Reisinger}
\address{Mathematical Institute, University of Oxford,  Woodstock Road, Oxford, OX2 6GG, United Kingdom,
\texttt{christoph.reisinger@maths.ox.ac.uk}}

\begin{abstract}
In this short paper, we consider discrete-time Markov chains on lattices as approximations to continuous-time diffusion processes.
The approximations can be interpreted as finite difference schemes for the generator of the process.
We derive conditions on the diffusion coefficients which permit transition probabilities to match locally first and second moments.
We derive a novel formula which expresses how the matching becomes more difficult for larger (absolute) correlations and strongly anisotropic processes, such that instantaneous moves to more distant neighbours on the lattice have to be allowed. Roughly speaking, for non-zero correlations, the distance covered in one timestep is proportional to the ratio of volatilities in the two directions.
We discuss the implications to Markov decision processes and the convergence analysis of approximations to Hamilton-Jacobi-Bellman equations in the Barles-Souganidis framework.
\end{abstract}
\bigskip

\maketitle

\hspace{-.45cm}\textit{Key Words:}
Two-Dimensional Diffusions;
Markov Chain Approximations;
Monotone Finite Difference Schemes
\bigskip

\hspace{-.45cm}\textit{2010 Mathematics Subject Classification:} 60J10, 60J60, 65C40, 65M06 \bigskip

\newtheorem{theorem}{Theorem}[section]
\newtheorem{prop}{Proposition}[section]
\newtheorem{la}[theorem]{Lemma}
\newtheorem{corollary}[theorem]{Corollary}
\newtheorem{remark}[theorem]{Remark}
\newtheorem{prob}[theorem]{Problem}	
\newtheorem{definition}[theorem]{Definition}

\section{Introduction}

The approximation of multi-dimensional continuous-time stochastic processes by Markov chains is a question of practical importance, which leads to a surprisingly complex answer.
Our interest is motivated by stochastic control problems, as they arise very prominently in mathematical finance and other application areas.

The computation of optimal control strategies usually relies on the computation of a value function, this being the expected value of an objective  under the optimal control strategy, as a function of the starting point. Approximation of this function by dynamic programming is usually either mesh-based or estimated by regression of conditional expectations 
over simulated trajectories. We focus here on the former approach.

A classical strategy consists in the approximation of the continuous-time stochastic process by a discrete-time Markov chain on a lattice, with a certain time step and mesh width of the lattice, as in the standard reference \cite{kushner2013numerical}. Then it can be shown, under relatively mild assumptions on the underlying diffusions, that if the moments are matched  locally in time and state space, asymptotically for vanishing time step and mesh width, the value of the control problem associated with the discrete process converges to the original value function.

These approximations are intimately linked to finite difference schemes for the Hamilton-Jacobi-Bellman PDEs which govern the value function. Indeed, a standard method of deriving suitable transition probabilities for the Markov chains is based on finite difference approximations to the generator of the process. The moment matching conditions translate into consistency of the corresponding finite difference schemes.
It is evident that only those finite difference schemes lead to viable Markov chain approximations, where the elements of the discretisation matrix, a scaled version of which are the candidate transition probabilities, are non-negative.

Barles and Souganidis \cite{barlessouganidis} prove that under certain mild regularity conditions and a comparison principle, a consistent, monotone and stable approximation scheme converges to the true viscosity solution of a second-order non-linear PDE. 
Since then, many papers have added to the analysis, including convergence orders, of monotone approximation schemes especially for equations of Hamilton-Jacobi-Bellman(HJB)-type, see, e.g., \cite{krylov97, bonnanszidani03, barlesjakobsen02, barlesjakobsen05, barlesjakobsen07, debrabantjakobsen12}.
The requirement of monotonicity is equivalent to a condition on positivity of the coefficients (see \cite{forsythlabahn}), thus demonstrating a one-to-one correspondence between those finite difference scheme with generally provable convergence, and Markov chain approximations.

The focus of this paper is the construction of transition probabilities, or -- equivalently -- monotone discretisation operators,
in the two-dimensional case. We now review prior results from the literature, often
presented in the context of a (degenerate) elliptic PDE and we translate this into our language.

It is already recognised in \cite{kushner2013numerical} that the approximation becomes more difficult when the off-diagonal terms of the covariance matrix are comparable to or larger than the diagonal terms. This is the case for large correlations and anisotropic variances.

More specifically, a result tracing back to \cite{motzkinwasow53} implies that a positive coefficient discretisation using only neighbouring nodes exists if and only if the covariance matrix is diagonally dominant.
Moreover, \cite{crandalllions96} show that if the problem is degenerate (correlation $\rho=1$), then the problem can be approximated locally, i.e., using transitions to mesh points which are a finite number of steps away, if and only if the ratio between the volatilities in both directions is rational.
This will also be directly visible from our analysis.
Such equations with perfect correlation arise, e.g., for stochastic control problems with partial information \cite[Chapter~6, {\S}10]{fleming2012deterministic}.
Specifically, \cite{monoyios2010utility} studies financial derivative hedging and valuation in incomplete markets with unknown (constant) drift of the asset price processes.
After application of the K{\'a}lm{\'a}n-Bucy filter, the process for the filtered drift is driven by the same Brownian motion as the asset price process, resulting in degenerate multi-dimensional diffusions. The present results can therefore be of interest in this setting.

In the case where monotone schemes exist in the irrational but non-degenerate case, \cite{kocan95} derives lower and upper bounds on the required minimal stencil size using Diophantine equations and approximation results for irrationals by rationals. Upper and lower bounds are given in terms of inverse fractional powers of the parameter $1-|\rho|$, demonstrating that the approximation schemes become necessarily and increasingly non-local as the problem degenerates. This feature is also visible in our analysis, however, the numerical illustrations will show that except for extreme correlations a very small number of neighbouring points suffices. In the case of $\rho=0.99$ and for moderate anisotropy, only 3 layers of cells in each direction are needed.

An alternative characterisation of monotone schemes via dual cones is given in \cite{bonnanszidani03}, and this is the work most closely related to the present paper. This method of analysis is constructive as it naturally leads to the discretisation coefficients as the solution of a linear program.
An efficient implementation is discussed in \cite{bonnansottenwaelterzidani04}. 
Where our analysis differs is that \cite{bonnanszidani03} formulate their approximations in terms of linear combinations of standard one-dimensional diffusion approximations in all possible directions to lattice points in a certain vicinity, while our primitive quantity are the transition probabilities directly. In so doing, we are able to condense the conditions in Section 5 of  \cite{bonnanszidani03} into a single formula which lends itself easily to numerical illustrations and can be simplified further to an easy to evaluate rule of thumb.


The novel contributions of this paper are:
\begin{itemize}
\item
a precise characterisation of the model parameters (variances and correlations) for which local Markov chain approximations on a lattice exist (i.e., the process moves not more than a given fixed number of nodes in every timestep), see Theorem \ref{theo:mono};
\item
equivalently, a characterisation of monotone discretisations of (degenerate) elliptic PDEs with fixed given stencil size;
\item
a very simple necessary condition, see Corollary \ref{cor:region};
\item
numerical illustrations of these theoretical results which show that the simple necessary condition from Corollary \ref{cor:region} is usually also sufficient.
\end{itemize}

The rest of this paper is structured as follows. In Section \ref{sec:prelims}, we define the notation, reduce the problem to the simplest representative setting and formulate a linear program the solution of which are the transition probabilities. Section \ref{sec:results} derives the main results on the existence of positive solutions by analysing the existence of solutions to a dual problem.
In Section \ref{sec:numerics}, we present and discuss numerical parameter studies of the formulae, and Section \ref{sec:conclusions} offers conclusions.


\section{Preliminaries}
\label{sec:prelims}

We start by considering a two-dimensional It{\^o} process of the form
\begin{eqnarray}
\label{dX1}
d X_{1,t} &=& \sigma_1(X_{1,t},X_{2,t},t) \, dW_{1,t}, \\
d X_{2,t} &=& \sigma_2(X_{1,t},X_{2,t},t) \, dW_{2,t},
\label{dX2}
\end{eqnarray}
where $(W_1,W_2)$ is a two-dimensional standard Brownian motion with local correlation $\rho$,
\begin{eqnarray*}
d[W_{1}, W_{2}]_t = \rho(X_{1,t},X_{2,t},t) \, dt,
\end{eqnarray*}
and $\sigma_1$ and $\sigma_2$ are local volatility functions.
We omit drift terms as their inclusion does not lead to any extra difficulty and does not alter the conclusions
(see Remark \ref{rem:upwind} later). 

The most practically relevant application is to control problems, where the coefficients also depend on a control parameter. For the purposes of this study, we do not explicitly include this dependence in the notation and it is understood that the approximation is for a given control.


\subsection*{Moment matching conditions}

We seek to approximate $X = (X_1,X_2)$ by a discrete-time Markov chain $\widehat{X} = (\widehat{X}(n))_{n\ge 0}$ on a lattice $(x_{1,i},x_{2,j}) = (i h, j H)$, $i,j\in\Z$, where $h,H>0$ are given mesh widths in the two directions. We interpret $\widehat{X}(n)$ as approximation to $(X_{1,t_n}, X_{2,t_n})$,
where $t_{n} = n k$ are uniformly spaced by some $k>0$.


Following \cite{kushner1990numerical}, we set up conditions by which $\widehat{X}$
asymptotically matches the moments of the increments of $X$, specifically,
\begin{eqnarray}
\label{expectation}
\mathbb{E}\left(\widehat{X}(n+1) - \widehat{X}(n) \vert \widehat{X}(n) = (x_{1,i},x_{2,j}) \right) &=& o(k),
\end{eqnarray}
for all $(x_{1,i},x_{2,j})$,
\begin{eqnarray*}
\mathbb{E}\left((\widehat{X}_1(n+1)-\widehat{X}_1(n))^2\vert \widehat{X}(n) = (x_{1,i},x_{2,j}) \right) &=& \sigma_1^2(x_{1,i},x_{2,j},t_n) \, k + o(k),
\end{eqnarray*}
similar for $\widehat{X}_2$, and
\begin{eqnarray*}
\mathbb{E}\left((\widehat{X}_1(n+1)-\widehat{X}_1(n))(\widehat{X}_2(n+1)-\widehat{X}_2(n))
\vert \widehat{X}(n) = (x_{1,i},x_{2,j}) \right) && \\
 & \hspace{-16 cm} =&\hspace{-8 cm}\rho(x_{1,i},x_{2,j},t_n) \sigma_1(x_{1,i},x_{2,j},t_n) \sigma_2(x_{1,i},x_{2,j},t_n) \, k + o(k).
\end{eqnarray*}
See also, e.g., 
\cite{kushner2013numerical}.
The convergence of the approximations is then shown in \cite{kushner1990numerical} for Lipschitz coefficients $\sigma$ and $\rho$.

The moments of $\widehat{X}$ are determined by the transition probabilities
\begin{eqnarray*}
 p_{i,j}^{l,m} &\equiv& \mathbb{P}\left(\widehat{X}(n+1) = (x_{1,i},x_{2,j}) \vert \widehat{X}(n) = (x_{1,l},x_{2,m}) \right),
\end{eqnarray*}
where we omit for brevity that $ p_{i,j}^{l,m}= p_{i,j}^{l,m}(k,h,H)$.

%

We can focus on a single point, $t=0$, $n=0$, $l=0$, $m=0$ to establish the matching conditions.
For ease of notation, we define
\begin{eqnarray*}
 p_{i,j} &\equiv& \mathbb{P}\left(\widehat{X}(1) = (x_{1,i},x_{2,j}) \vert \widehat{X}(0) = (0,0) \right).
\end{eqnarray*}

The first moment condition (\ref{expectation}) becomes
\begin{eqnarray*}
\sum_{i,j=-\infty}^\infty i h p_{ij} = o(k)
\qquad
\text{and}
\qquad
\sum_{i,j=-\infty}^\infty j H p_{ij} = o(k).
\end{eqnarray*}

We assume now that $\sigma_1(0,0,0), \sigma_2(0,0,0)>0$. The case where exactly one coefficient is zero reduces to a one-dimensional problem and the standard approximation to the Brownian driver using a simple symmetric random walk is admissible; the case where both coefficients are zero is trivial. These statements still hold in the case with drift (see Remark \ref{rem:upwind}).
By a suitable rescaling of the variables, we can then reduce the equations to the special case $\sigma_1(0,0,0) \equiv 1$,
$\sigma_2(0,0,0) \equiv \sigma$, $\rho(0,0,0) \equiv \rho$, where $\sigma>0$. Then, similar to above, we need for the second moments
\begin{eqnarray*}
\sum_{i,j=-\infty}^\infty (i h)^2 p_{ij} &= & k + o(k), \\
\sum_{i,j=-\infty}^\infty (j H)^2 p_{ij} &=& \sigma^2 k + o(k), \\
\sum_{i,j=-\infty}^\infty (i h)(j H) p_{ij} &=& \rho \sigma k + o(k).
\end{eqnarray*}

There is some arbitrariness in the choice of $k$, $h$ and $H$. We fix a refinement regime where
\begin{eqnarray}
\label{meshrat}
\frac{k}{h^2} = 1
\end{eqnarray}
fixed and, with $\sigma>0$, keep the ratio
\begin{eqnarray}
\label{rat}
\rat &\equiv& \frac{1}{\sigma} \frac{H}{h}
\end{eqnarray}
also fixed as a measure of the anisotropy of the discretised problem.
We re-iterate that $R$ is a local quantity which depends on $x_1$ and $x_2$ if $\sigma$ is not constant. If $\sigma_i$ depends only on $x_i$ for $i=1,2$, it is possible to adapt $h$ and $H$ as a function of $x_1$ and $x_2$, respectively, and hence construct a tensor product mesh with effectively constant $R$. In the general case, where one of the $\sigma_i$ depends non-trivially on both $x_1$ and $x_2$, such a procedure will not be possible.

It will be seen that $\rat$ and $\rho$ are the key parameters for the locality of the processes.
If we chose a different refinement strategy (not keeping (\ref{meshrat}) and (\ref{rat}) constant as $k\rightarrow 0$), this would change the actual probabilities but only by a scaling factor (for $(i,j)\neq (0,0)$).

For computations, it is necessary to restrict the possible transitions to a finite set, and for computational efficiency it will be advantageous to choose this set as small as possible.
We denote by some $0<s \in \mathbb{N}$ the largest distance in any direction between a pair of nodes with non-zero transition probability.

Summarising, the moment matching conditions for $k\rightarrow 0$ are
\begin{eqnarray}
\label{second-order}
\hspace{-0.7 cm}
&& \sum_{i,j=-s}^s i^2 p_{ij} = 1, \qquad \qquad  \!\!\! \sum_{i,j=-s}^s j^2 p_{ij} = \rat^{-2}, \qquad  \sum_{i,j=-s}^s i j p_{ij} = \rat^{-1} \rho, \\
\label{first-order}
\hspace{-0.7 cm}&& \sum_{i,j=-s}^s i p_{ij} = 0,  \qquad \qquad \sum_{i,j=-s}^s j p_{ij} = 0,  \label{first-order-last} \\
\label{zero-order}
\hspace{-0.7 cm}&& \sum_{i,j=-s}^s p_{ij} = 1.
\end{eqnarray}
These are $4s+6$ equations for $(2s+1)^2$ unknowns in total. See also \cite{crandalllions96}.

\subsection*{Relation to finite difference schemes}

The generator of the process (\ref{dX1}), (\ref{dX2}) is given by the elliptic differential operator
\begin{eqnarray}
\label{modelpde}
\mathcal{L} u \equiv \harf u_{x_1 x_1} + \rho \sigma u_{x_1 x_2} + \harf \sigma^2 u_{x_2 x_2}. 
\end{eqnarray}

We consider
a finite difference discretisation of (\ref{modelpde}), written as
\begin{equation}
\label{modelscheme}
(L_h u_h)_{n,m} \equiv
\frac{1}{h^2} \left( \sum_{\tiny \begin{array}{c} i,j =-s \\ (i,j)\neq (0,0) \end{array}}^s 
p_{ij} u_{n+i,m+j} - 
u_{0,0}  \!\!\! \sum_{\tiny \begin{array}{c} i,j =-s \\ (i,j)\neq (0,0) \end{array}}^s 
p_{ij} \right)
\end{equation}
for some $0<s \in \mathbb{N}$, on a mesh with widths $h$ and $H$ in the two directions.
We omit $H$ in the notation and assume it is picked as a function of $h$ (see also (\ref{rat})).
Then $u_{n,m}$ is a numerical approximation to $u(n h, m H)$ and we write $u_h=(u_{n,m})_{n,m\in\Z}$, also $L_h: \R^{\Z^2}\rightarrow \R^{\Z^2}$. This includes all linear schemes which use $(2s+1)^2$ mesh points in a rectangle around the central node to approximate $\mathcal{L}$.

The scheme (\ref{modelscheme}) is consistent with the PDE (\ref{modelpde}) if and only if for all smooth $\phi$
\begin{eqnarray}
\label{consistency}
\lim_{h,H\rightarrow 0} 
\Big( \mathcal{L} \phi - L_h \phi
\Big) = 0,
\end{eqnarray}
where
\[
L_h \phi
=
\frac{1}{h^2} \left( \sum_{\tiny \begin{array}{c} i,j =-s \\ (i,j)\neq (0,0) \end{array}}^s 
p_{ij} \phi(\cdot+i h,\cdot + j H) - 
\phi \!\!\! \sum_{\tiny \begin{array}{c} i,j =-s \\ (i,j)\neq (0,0) \end{array}}^s 
p_{ij} \right).
\]
By Taylor expansion, this is seen to be equivalent to the second, first and zero order conditions (\ref{second-order})--(\ref{zero-order}).

The scheme  (\ref{modelscheme}) is monotone in the standard sense (see \cite{forsythlabahn}) 
if and only if it has positive ``off-diagonal'' coefficients , 
\begin{equation}
\label{nonnegativity}
p_{ij} \ge 0, \qquad -s\le i, j\le s, \quad (i,j) \neq (0,0).
\end{equation}


The difficulty in constructing monotone discretisations arises from the cross-derivative term in (\ref{modelpde}) which is present for $\rho \neq 0$.
In this case, a standard discretisation using an iterated central difference is never monotone.
In \cite{sulem2006}, p.~154 of Section 9.4, a scheme is given which takes into account the sign of the covariance terms of multi-dimensional diffusions, here determined by the sign of $\rho$:
\begin{eqnarray*}
\frac{\partial^2 \phi}{\partial x_1 \partial x_2}(x_1,x_2) &\;\sim\;& \pm \
\frac{2 \phi(x_1,x_2) + \phi(x_1+h,x_2\pm H)  + \phi(x_1-h,x_2\mp H)}{ h \ H}
\\
&\;\;& \hspace{-1.8 cm} \mp \ \frac{\phi(x_1+h,x_2) + \phi(x_1-h,x_2) + \phi(x_1,x_2+H) + \phi(x_1,x_2-H)}{h \ H},
\end{eqnarray*}
where the top signs refer to the case $\rho\ge 0$ and the bottom ones to $\rho\le 0$. We denote by $\sim$ that the discretisation is consistent with the mixed partial derivative, i.e., for smooth $\phi$ the right-hand side converges to the left-hand expression as $h,H\rightarrow 0$.

The resulting scheme for (\ref{modelpde}), using the above \emph{seven-point stencil} for the cross-derivative and standard central differences for the second derivatives in $x_1$ and $x_2$ directions, is easily seen to be consistent overall,
and monotone for given $\sigma$ and  $\rho$ provided that $h$ and $H$ are chosen such that $|\rho| \le \min(\rat,1/\rat)$ with $\rat$ as
per (\ref{rat}).
A more general condition on monotonicity in the multi-dimensional drift-diffusion setting is given in \cite{sulem2006}.

%

\begin{remark}
\label{rem:upwind}
If the equations (\ref{dX1}) and (\ref{dX2}) contain drift terms, an upwind discretisation can be used to construct transition probabilities.
For simplicity, we focus on the case where instead of (\ref{dX1})  we have
\begin{eqnarray*}
d X_{1,t} &=& \mu_1(X_{1,t},X_{2,t},t) \, dt + \sigma_1(X_{1,t},X_{2,t},t) \, dW_{1,t} 
\end{eqnarray*}
but without drift in (\ref{dX2}).
Without loss of generality, we study again the point $(0,0,0)$, and assume $\mu \equiv \mu_1(0,0,0)\ge 0$. Then, for sufficiently small $k$,
\begin{eqnarray*}
\widehat{p}_{10} \; =\;   p_{10} + \mu \frac{k}{h}, \qquad
\widehat{p}_{00} \; =\;   p_{00} - \mu \frac{k}{h}, \qquad
\widehat{p}_{ij} \; = \; p_{ij}, \;\; \text{\small $(i,j) \notin \{(0,0),(0,1)\}$}
\end{eqnarray*}
are positive transition probabilities which are consistent with the generator in the sense of (\ref{consistency}), and hence satisfy the moment matching conditions.
Here, $(p_{ij})$ is a solution to the drift-free problem with arbitrary diffusion coefficients (possibly zero).
The case $\mu<0$ and with multiple drifts is similar.
\end{remark}


\section{Main results}\label{Section_Proof}
\label{sec:results}

We show the following result, 
which gives a uniform treatment of the conditions listed in Section 5 of \cite{bonnanszidani03}.
The proof is based on duality, which allows us to turn  equality constraints for a large number of primal variables into a single inequality for a single variable.


\begin{theorem}
\label{theo:mono}
There is a non-negative 
solution to (\ref{second-order})--(\ref{zero-order}) 
if and only if
\begin{eqnarray}
\label{geninequ}
\inf_{0<z_1<\frac{2|\rho|}{\rat}} \left( \rat^2 z_1 + \max_{\xi\in \s} \left(2\xi - \xi^2 z_1\right) \right) &\ge& 2 \rat |\rho|,
\end{eqnarray}
where $\s=\{i/j: 1\le i,j \le s\}$.
\end{theorem}
\begin{proof}
The proof uses Farkas' Lemma (see, e.g., \cite{nocedalwright}, p.~327), which ascertains for $b\in \mathbb{R}^m$, $A\in \mathbb{R}^{m\times n}$ that
\begin{eqnarray*}
\exists \; x \in \mathbb{R}^n: \quad 
Ax=b, \; x\ge0 
\qquad \Leftrightarrow \qquad 
\nexists \; y \in \mathbb{R}^m: \quad 
A^T y \ge 0, \; b^T y < 0.
\end{eqnarray*}

We first observe that we can drop equation (\ref{zero-order}) because if there is a non-negative solution vector $(p_{ij})$ with $(i,j)\neq(0,0)$ to
(\ref{second-order}) and (\ref{first-order}), it follows from (\ref{second-order}) that
\[
\sum_{(i,j) \neq (0,0)} p_{ij} \le \sum_{(i,j) \neq (0,0)} i^2 p_{ij} = \sum_{(i,j)} i^2 p_{ij}  = 1,
\]
and hence we can find $p_{00}\ge 0$ to satisfy (\ref{zero-order}), without changing the other equalities.

For given $\rat$, $\rho$ and $s$, instead of analysing the existence of a solution to (\ref{second-order})--(\ref{zero-order}) and (\ref{nonnegativity}) it is hence equivalent to analyse the existence of a $y \in \mathbb{R}^5$ such that
\begin{eqnarray}
\label{firstineq}
i^2 y_1 + j^2 y_2 + i j y_3 + i y_4 + j y_5 &\ge& 0 \quad \forall -s\le i,j\le s, \\
y_1 +  \rat^{-2} y_2 + \rat^{-1} \rho y_3 &<& 0.
\label{secondineq}
\end{eqnarray}

Denote the first expression by $p(\rat;i,j;y_1,y_2,y_3,y_4,y_5)$, 
then by comparing $p(\rat;i,j;\ldots)$ and $p(\rat;-i,-j;\ldots)$ it is seen directly that
\[
\max_{y_4,y_5} \min_{i,j} p(\rat;i,j;y_1,y_2,y_3,y_4,y_5) = \min_{i,j} p(\rat;i,j;y_1,y_2,y_3,0,0).
\]
So we can take $y_4=y_5=0$ in (\ref{firstineq}), and
setting $j=0$, it follows $y_1\ge 0$, and similarly $y_2\ge 0$. \\

For simplicity, assume now $\rho>0$ (the case $\rho<0$ is similar). \\

Then, from (\ref{secondineq}), $y_3 < 0$.
Dividing (\ref{firstineq}) and (\ref{secondineq}) through by $|y_3|$, 
defining $z_1 = 2 y_1/|y_3|$, $z_2=2 y_2/|y_3|$,
there is a solution to (\ref{firstineq}), (\ref{secondineq}) if and only if there is a solution $z_1, z_2\ge 0$ to
\begin{eqnarray}
\label{firstsimple}
i^2 z_1 + j^2 z_2 - 2 i j &\ge& 0 \quad \forall -s\le i,j\le s, \\
\rat^2 z_1 -2 \rho \rat + z_2 &<& 0.
\end{eqnarray}
The inequality (\ref{firstsimple}) is always true if $z_1,z_2\ge 0$ and $j=0$ or $i=0$. For $j>0, i\ge 0$, it is easy to see (dividing by $j^2$ and checking the signs of individual terms) that (\ref{firstsimple}) is equivalent to
\begin{eqnarray}
\label{secondsimple}
\min_{1\le i,j \le s} \left[ (i/j)^2 z_1 - 2 (i/j) + z_2\right] \,\ge\, 0 \quad \!\! \Leftrightarrow \quad \!\! \min_{\xi \in \s} \; \left[\xi^2 z_1 - 2 \xi + z_2\right] \ge 0, \hspace{-1cm}
\end{eqnarray}
where $\s=\{i/j: 1\le i,j \le s\}$. The same conclusion holds by a similar argument for $i>0, j\ge 0$, and for the general case we note that the inequalities for $(-i,-j)$ are identical to $(i,j)$, and the left-hand side of (\ref{firstsimple}) for $ij<0$ is larger than for $ij>0$. 
So we are looking for $z_1, z_2\ge 0$ such that
\begin{eqnarray}
\label{gencond}
- \min_{\xi\in \s} (\xi^2 z_1 - 2 \xi) \ \le \ z_2 \ < \ - \rat^2 z_1 + 2 \rat \rho. 
\end{eqnarray}
For the right-hand side to be positive it is needed that
\[
z_1 < \frac{2 \rho}{R},
\]
and the left-hand side is smaller than the right-hand side if
\[
\rat^2 z_1 - \min_{\xi\in \s} (\xi^2 z_1 - 2 \xi) = \rat^2 z_1 + \max_{\xi\in \s} \left(2\xi - \xi^2 z_1\right) < 2 \rat \rho.
\]
Both inequalities can simultaneously be satisfied by some $z_1$ if and only if
\begin{eqnarray}
\label{nonex}
\inf_{0<z_1<\frac{2|\rho|}{\rat}} \left( \rat^2 z_1 + \max_{\xi\in \s} \left(2\xi - \xi^2 z_1\right) \right) &<& 2 \rat \rho.
\end{eqnarray}

The case $\rho<0$ is similar and the statement follows. 

%
%
\end{proof}


\begin{corollary}
\label{cor:region}
For $s< |\rho| \max(\rat,1/\rat)$, there is no solution to (\ref{second-order})--(\ref{nonnegativity}).
\end{corollary}
\begin{proof}
Letting $z_1$ go to zero in (\ref{gencond}), the expression on the left approaches $2s$ and the expression on the right $2 \rat \rho$, hence
the chain of inequalities can always be satisfied with some $z_2>0$ if $s<\rat \rho$. The result follows by the symmetry between $\rho$ and $-\rho$ and $R$ and $1/R$.
\end{proof}

\section{Discussion and numerical illustrations}
\label{sec:numerics}

We now discuss the values for $\rat$ and $\rho$ which satisfy (\ref{nonex}), i.e.\ the converse of (\ref{geninequ}),  such that no viable approximation exists.

By completing the square one sees that
\[
\max_{\xi\in \s} \left(2\xi - \xi^2 z_1\right) = 
\max_{\xi\in \s}\left( 
 \frac{1}{z_1} - \left(\xi \sqrt{z_1} - {1}/{\sqrt{z_1}} \right)^2
 \right)
\le \frac{1}{z_1},
\]
with equality if $z_1 \in \s$, with $\s$ defined in Theorem \ref{theo:mono}. Similarly,
\[
\inf_{z_1>0} \left( \rat^2 z_1 + \max_{\xi\in \s} \left(2\xi - \xi^2 z_1\right) \right) \le
\inf_{z_1>0}
 \left(\rat \sqrt{z_1} - 1/\sqrt{z_1}\right)^2 + 2 \rat, 
\]
where the 
right-hand-side equals $2 \rat$ if and only if $z_1 = 1/\rat$.
From this we see that a monotone discretisation exists for all $\rho$ if and only if $\rat \in \s$, in particular $\rat$ rational (see also \cite{crandalllions96,kocan95}).

The maximum over $\xi$ in (\ref{nonex}) is attained at values of $\xi(z_1)$ which are piecewise constant in $z_1$ ($\s$ is a finite set).
For such intervals of $z_1$ in which $\xi(z_1)$ is constant at some $\xi$, from (\ref{nonex})
\begin{eqnarray}
\label{repinequ}
\inf_{z_1>0} \left( \rat^2 z_1 -2 \rat |\rho| + 2\xi - \xi^2 z_1 \right) &<& 0,
\end{eqnarray}
so for any pair $(z_1,\rat)$ for which the bracketed expression is negative, this is also true for $z_1$ in an interval $(z_-(\rat),z_+(\rat))$.
These regions are shown in Figure \ref{fig:rzplots}, left, for $s=1,3,5$.

Conversely, for given $\rat$, there is a solution $(z_1,z_2)$ to the dual problem and hence no solution to the primal problem,
if a horizontal line through $(0,\rat)$ intersects any of these regions. These lines are also shown in the plot, and correspond to the
intersections of the horizontal lines with the white areas on the right-hand-side plots.

It is also clear from (\ref{geninequ}) that for each $\rat$ there is a value $0\le \rho_{max}(\rat) \le 1$ such that a monotone discretisation exists for all $|\rho|\le \rho_{\max}(\rat)$.
We denote the set of all these pairs by
\[
\mathcal{\rat}=\{(\rat,\rho): \ -\rho_{\max}(\rat) \le \rho \le \rho_{\max}(\rat)\},
\]
shown  in the right-hand column of Figure \ref{fig:rzplots}, for $s=1$, $s=3$ and $s=5$ fixed.
By symmetry, it holds that $\rho_{\max}(1/\rat) = \rho_{\max}(\rat)$ for all $\rat>0$; we therefore only show the range $\rat\ge 1$ and $\rho\ge 0$.

\begin{figure}
\includegraphics[width=0.495\columnwidth]{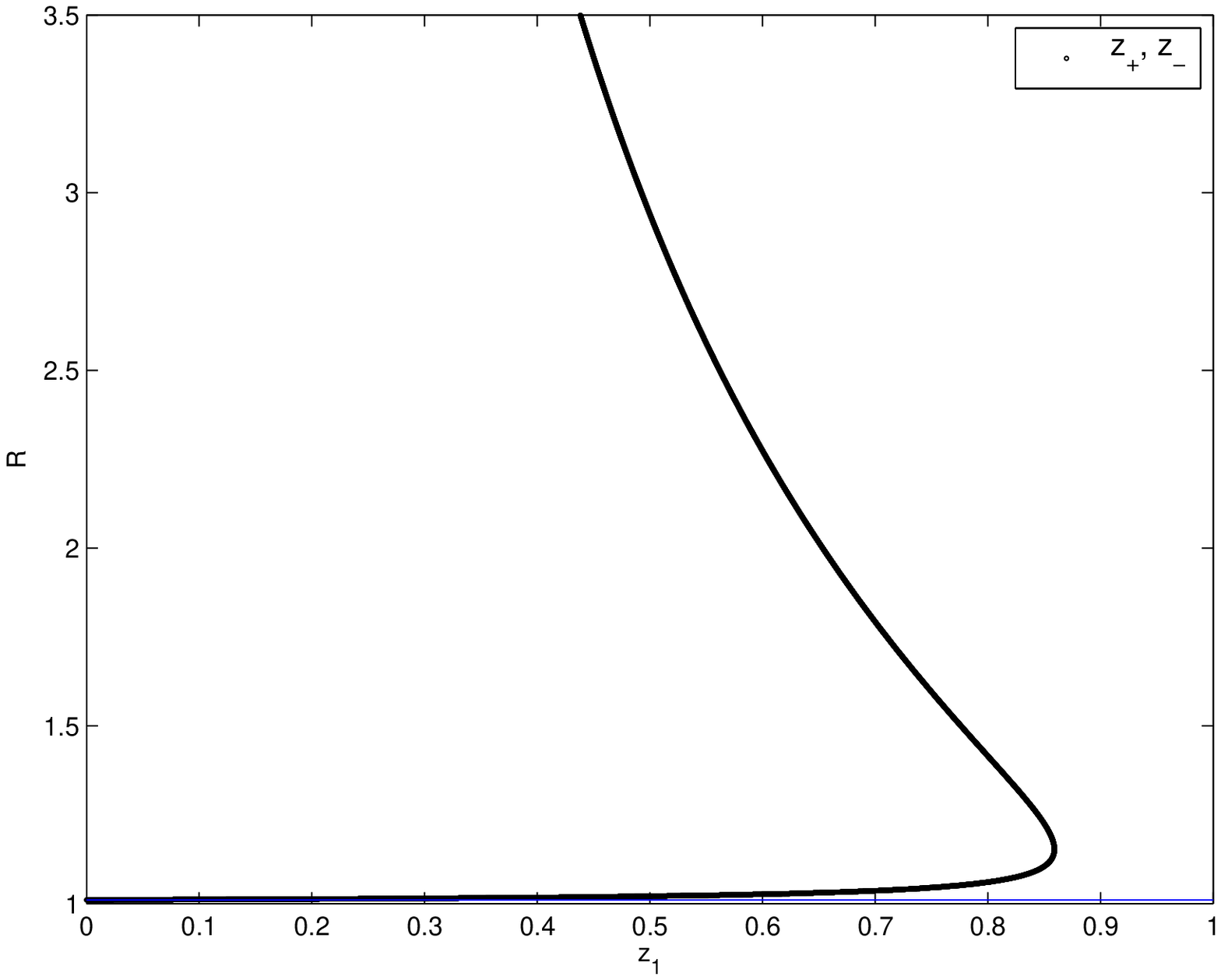} \hfill
\includegraphics[width=0.495\columnwidth]{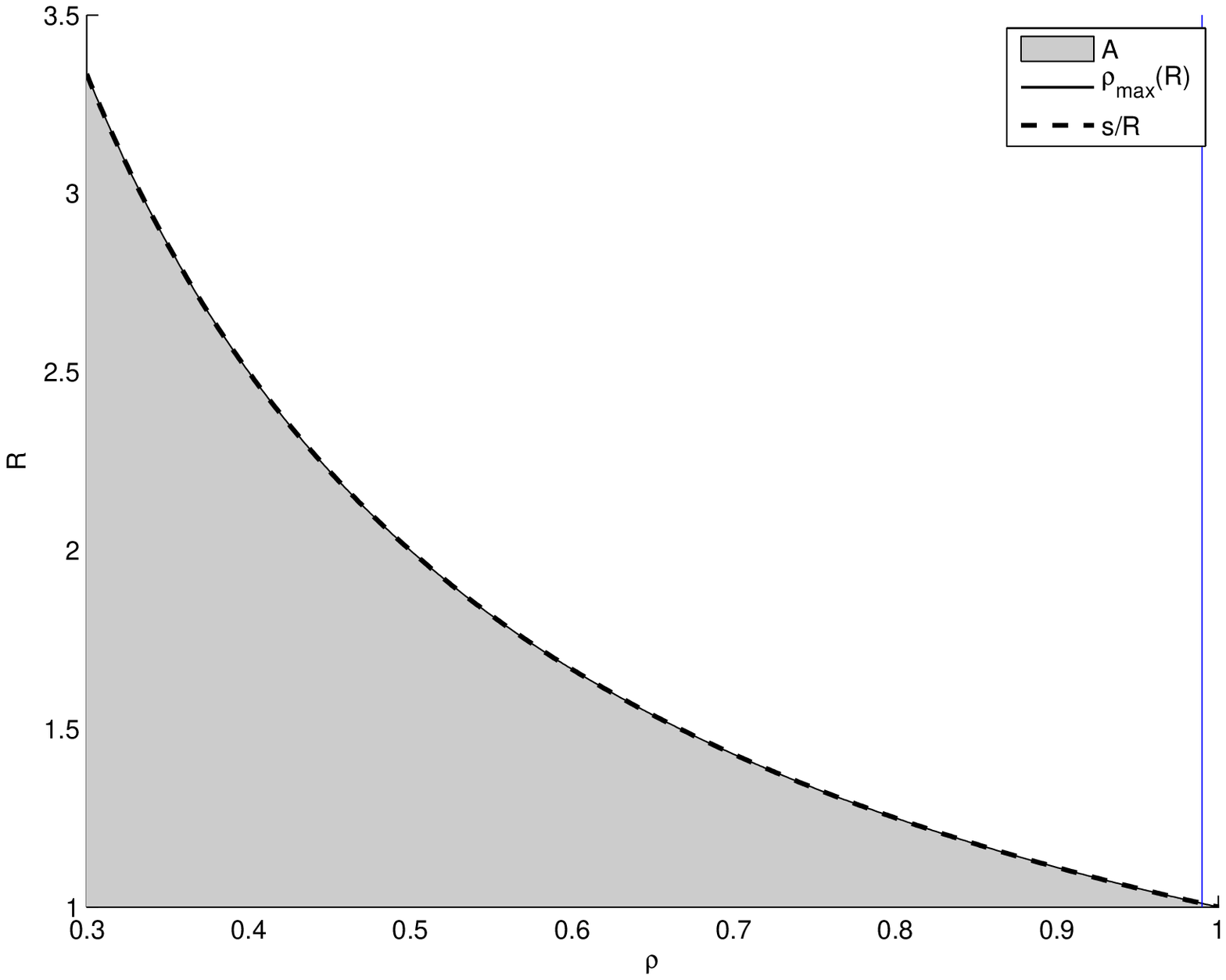} \\
\includegraphics[width=0.495\columnwidth]{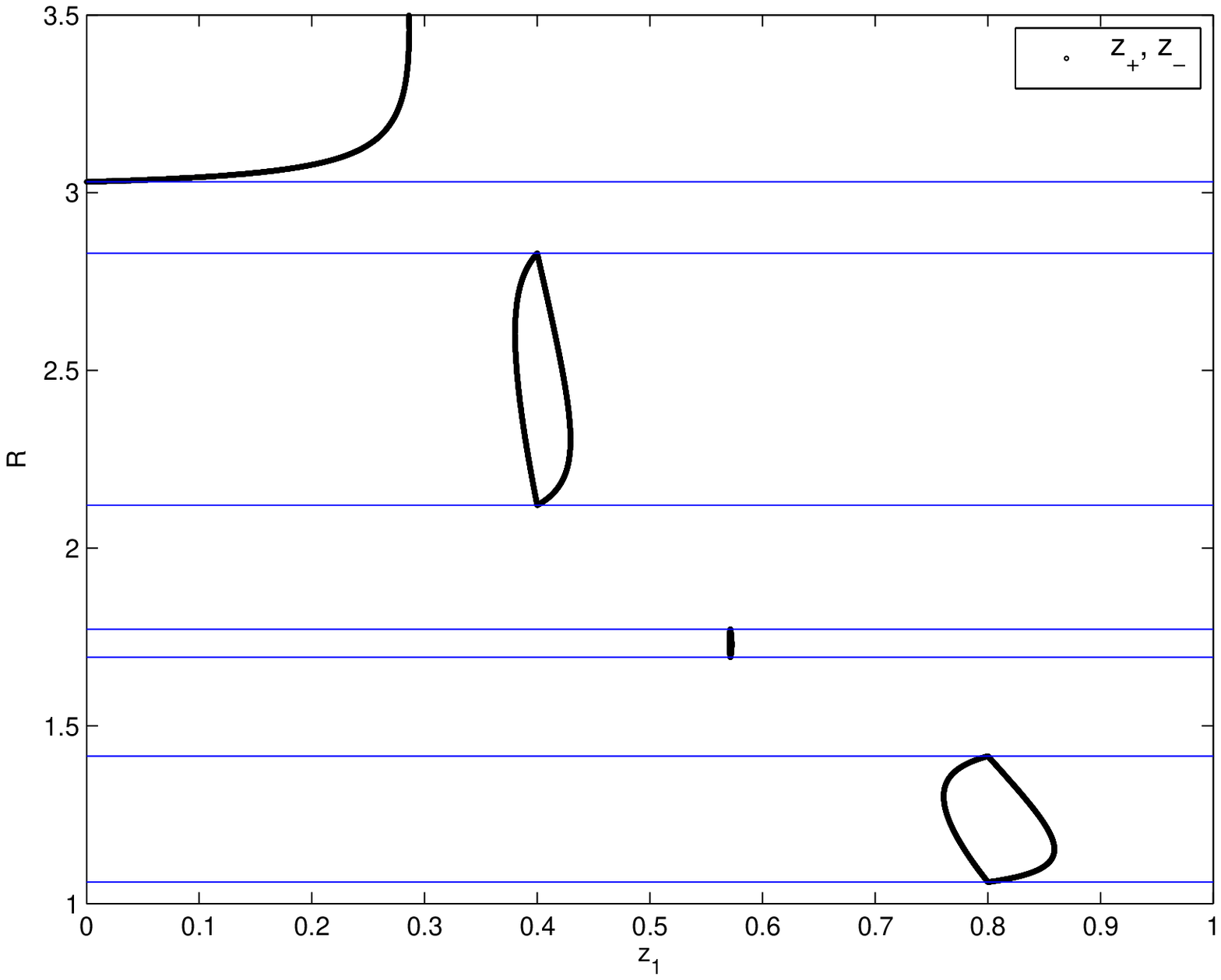} \hfill
\includegraphics[width=0.495\columnwidth]{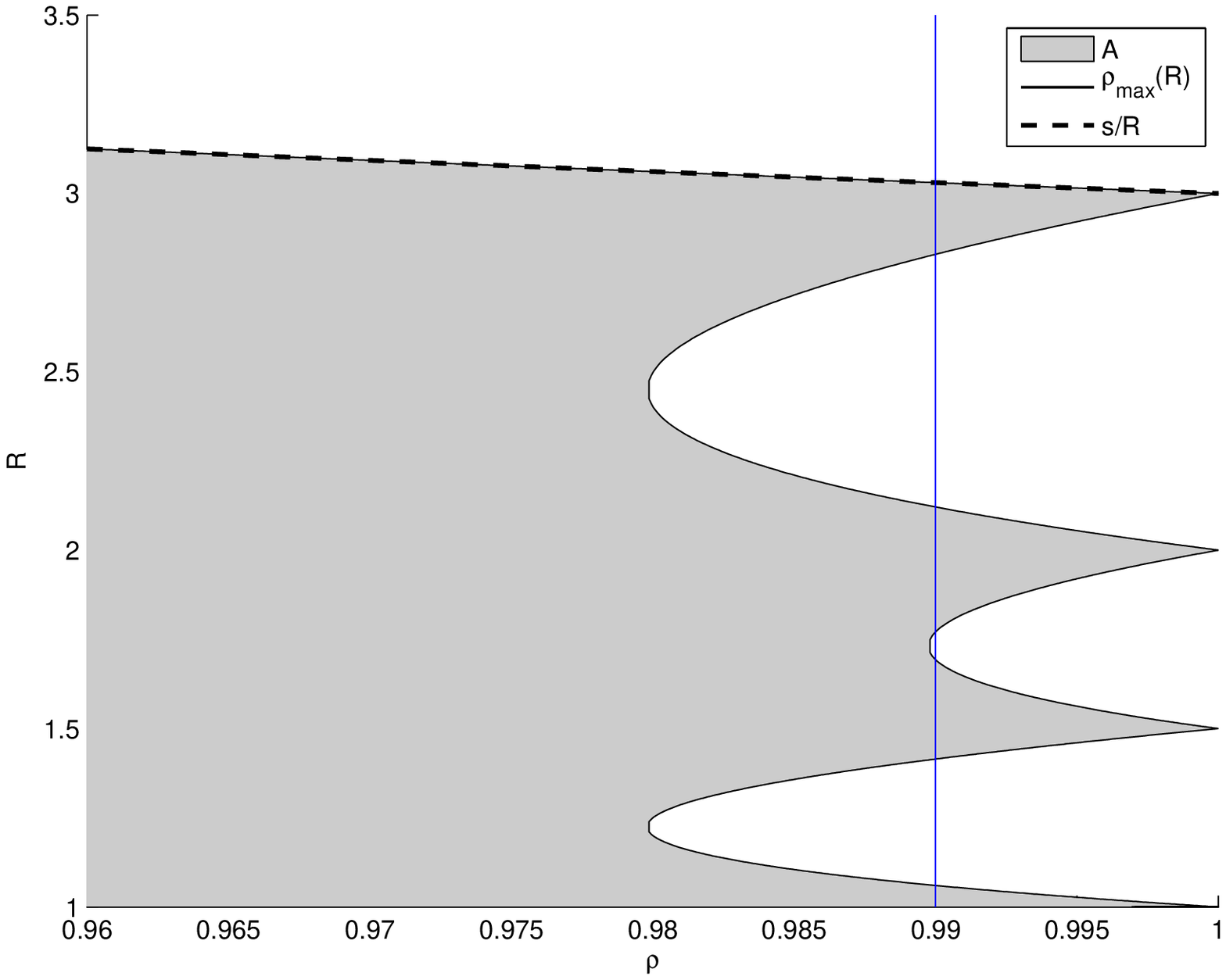} \\
\includegraphics[width=0.495\columnwidth]{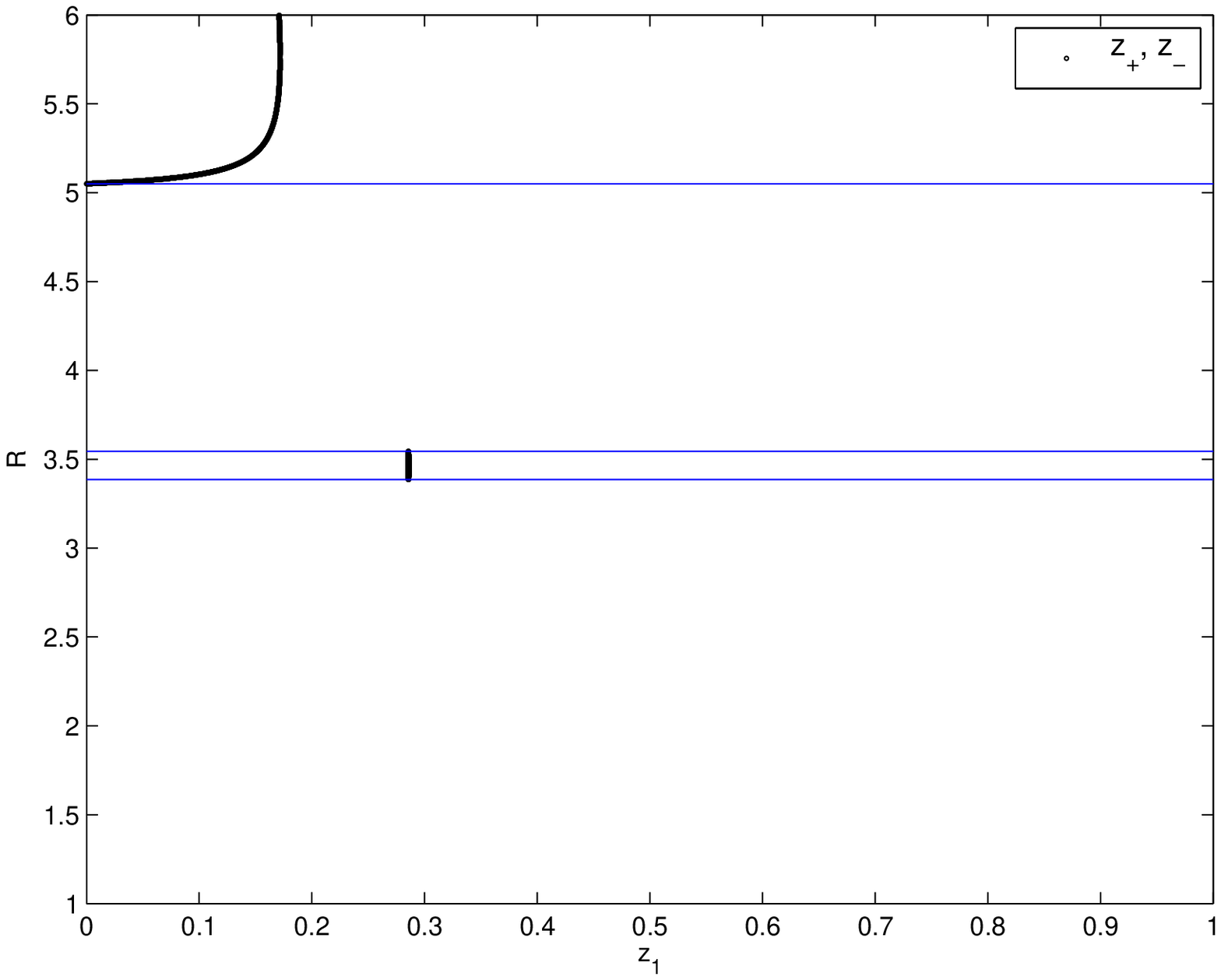} \hfill
\includegraphics[width=0.495\columnwidth]{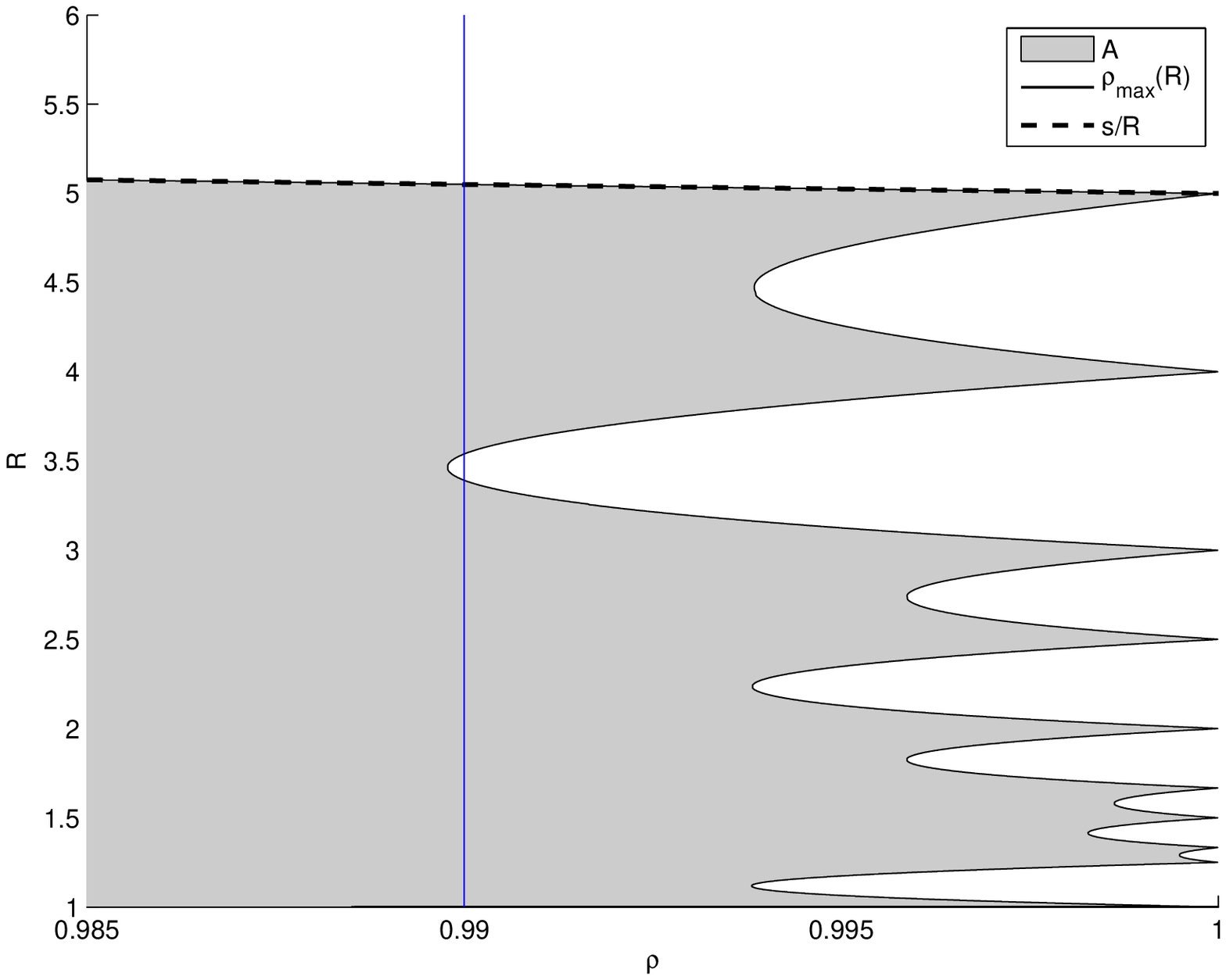}
\caption{Left: 
Given $R$, for values of $z_1$ between  $z_-(R)$ and $z_+(R)$ the dual problem has a solution (see the discussion following (\ref{repinequ})),
for $\rho=0.99$ and $s=1$ (top), $s=3$ (middle), and $s=5$ (bottom).
Right: In the same order for $s=1,3,5$, the region $A$ of pairs $(\rho,\rat)$, which is the subgraph of $\rho_{\max}$, for which the primal problem has a solution. The dashed curve $s/R$ is the simple upper bound from Corollary \ref{cor:region}, while the situation for $R<s$ is more complicated and fully described by (\ref{geninequ}).
The vertical lines at $\rho=0.99$ in the right plots intersect the shaded area at intervals which are reproduced by the horizontal lines in the left-hand plots, tangential to the curves $z_-$ and $z_+$.}
\label{fig:rzplots}
\end{figure}


The only values of $R$ for which there is a positive solution for all $\rho \in [-1,1]$ are $\rat\in \s$.
For example, in the bottom right plot for $\rat=5$ this is seen by the spikes at
\[
\rat \in \{1,5/4,4/3,3/2,5/3,2,5/2,3,4,5\}.
\]
For such $\rat=i/j \in \s$, where $i$ and $j$ are integers no larger than $s$, a discretisation that has these properties is the seven-point stencil applied to the nodes
which are $i$ steps away from the centre node in one direction and $j$ in the other.
For general $\rat$, the seven-point stencil loses monotonicity for certain values of $\rho$ (close to $\pm 1$).
The above proof and illustration show that any other local stencil must suffer the same shortcomings (see the result from \cite{crandalllions96} mentioned above).

This is unlikely to be a problem in practice for moderate $\rat$, $1/s \le \rat \le s$. For instance, for $s\ge 3$, a positive coefficient discretisation exists for $\rho$ up to a value close to one (precisely, 0.98 for $s=3$), 
and this value grows as $s$ increases.


For very anisotropic problems with $\rat>s$, however, we see that there is a sharp restriction $\rho\leq s/\rat$ on allowable $\rho$ (see Corollary \ref{cor:region} and Figure \ref{fig:rzplots}, right column). 

Consider, for instance, a finite difference discretisation of the two-dim\-en\-sion\-al Black-Scholes PDE (with zero interest rates),
\begin{eqnarray}
\label{bspde}
\frac{\partial V}{\partial t} +
\harf  \sigma_1^2 S_1^2 \frac{\partial^2 V}{\partial S_1^2} +
\rho \sigma_1 \sigma_2 S_1 S_2 \frac{\partial^2 V}{\partial S_1 \partial S_2} +
\harf  \sigma_2^2 S_2^2 \frac{\partial^2 V}{\partial S_2^2} 
= 0
\end{eqnarray}
for $S_1,S_2,t>0$,
on a uniform isotropic mesh, $h=H$. Then, locally, $\rat=(\sigma_1 S_1)/(\sigma_2 S_2)$. 
With mesh coordinates for $S_1$ and $S_2$ in $\{h, 2h,\ldots, S_{max}\}$,
$R$ ranges from $O(h)$ to $O(1/h)$.
This shows that extreme $\rat$ can happen naturally for locally degenerate equations.

To re-iterate, for the two-dimensional Black-Scholes PDE there is no consistent and monotone discretisation with fixed stencil on a uniform mesh, for \emph{any} $\rho\neq 0$.

\section{Conclusions}
\label{sec:conclusions}

We have studied the required step width of Markov chain approximations, or -- equivalently -- the necessary width of monotone finite difference stencils, to approximate two-dimensional diffusions.
The analysis in the previous sections shows that this width has to be at least $|\rho| R$, where $R$ is a measure of the anisotropy of the problem. As seen from the shaded areas in Figure \ref{fig:rzplots}, this condition is almost sufficient except when $|\rho|$ is extremely close to 1.

As the difficulty arises from the correlations (cross-derivatives) between two variables, the two-dimensional case highlights problems which also exist, in exacerbated form, in higher-dimensional settings. Indeed, the seven-point stencil is monotone only because negative matrix entries resulting from the cross-derivative between $x_1$ and $x_2$ are compensated by positive entries resulting from $x_1$ and $x_2$ derivatives. In higher dimensions, where there are multiple cross-derivatives involving $x_1$, say, the viable range of $\rho$ shrinks with increasing dimension,
but we do not analyse this further here.


It follows that the only generally applicable 
schemes require $s$ to depend on $h$, specifically $s(h)\rightarrow \infty$ as $h\rightarrow 0$. An example are the ``linear interpolation semi-Lagrangian'' (LISL) schemes of \cite{camilli1995approximation, debrabantjakobsen12}. There, nodes at a distance $k$ are involved in the discretisation where $h=o(k)$, such that $s\rightarrow \infty$ as $h\rightarrow 0$. The schemes are consistent even if  $s\rightarrow \infty$ arbitrarily slowly, although for optimal accuracy $s=O(h^{-1/2})$.
As our analysis shows, this degree of generality is only required in exceptional cases. It would therefore seem advantageous to switch to this wide stencil scheme only where a standard compact stencil fails to produce positive coefficients. A switching strategy based on this principle is proposed in \cite{ma2014unconditionally}, and their empirical finding for an uncertain volatility model is that in the majority of nodes the standard scheme can be selected, in line with our results.

\bibliography{compact}

\begin{thebibliography}{10}

\bibitem{barlesjakobsen02}
G.~Barles and E.~R. Jakobsen.
\newblock On the convergence rate of approximation schemes for
  {H}amilton-{J}acobi-{B}ellman equations.
\newblock {\em ESAIM:M2AN}, 36(1):33--54, 2002.

\bibitem{barlesjakobsen05}
G.~Barles and E.~R. Jakobsen.
\newblock Error bounds for monotone approximation schemes for
  {H}amilton-{J}acobi-{B}ellman equations.
\newblock {\em SIAM J. Numer. Anal.}, 43(2):540--558, 2005.

\bibitem{barlesjakobsen07}
G.~Barles and E.~R. Jakobsen.
\newblock Error bounds for monotone approximation schemes for parabolic
  {H}amilton-{J}acobi-{B}ellman equations.
\newblock {\em Math. Comp.}, 76(240):1861--1893, 2007.

\bibitem{barlessouganidis}
G.~Barles and P.~E. Souganidis.
\newblock Convergence of approximation schemes for fully nonlinear second order
  equations.
\newblock {\em Asympt. Anal.}, 4(3):271--283, 1991.

\bibitem{bonnansottenwaelterzidani04}
F.~Bonnans, {\'E}.~Ottenwaelter, and H.~Zidani.
\newblock A fast algorithm for the two dimensional {HJB} equation of stochastic
  control.
\newblock {\em ESAIM:M2AN}, 38(4):723--735, 2004.

\bibitem{bonnanszidani03}
F.~Bonnans and H.~Zidani.
\newblock Consistency of generalized finite difference schemes for the
  stochastic {HJB} equation.
\newblock {\em SIAM J. Num. Anal.}, 41(3):1008--1021, 2003.

\bibitem{camilli1995approximation}
F.~Camilli and M.~Falcone.
\newblock An approximation scheme for the optimal control of diffusion
  processes.
\newblock {\em Mod{\'e}lisation math{\'e}matique et analyse num{\'e}rique},
  29(1):97--122, 1995.

\bibitem{crandalllions96}
M.~G. Crandall and P.-L. Lions.
\newblock Convergent difference schemes for nonlinear parabolic equations and
  mean curvature motion.
\newblock {\em Numer. Math.}, 75(1):17--41, 1996.

\bibitem{debrabantjakobsen12}
K.~Debrabant and E.~R. Jakobsen.
\newblock Semi-{L}agrangian schemes for linear and fully non-linear diffusion
  equations.
\newblock {\em Math. Comp.}, 82(283):1433--1462, 2013.

\bibitem{fleming2012deterministic}
W.~H. Fleming and R.~W. Rishel.
\newblock {\em Deterministic and stochastic optimal control}.
\newblock Springer, 2012.

\bibitem{forsythlabahn}
P.~A. Forsyth and G.~Labahn.
\newblock Numerical methods for controlled {H}amilton-{J}acobi-{B}ellman {PDE}s
  in finance.
\newblock {\em J. Comp. Fin.}, 11(2):1--44, 2007/2008.

\bibitem{kocan95}
M.~Kocan.
\newblock Approximation of viscosity solutions of elliptic partial differential
  equations on minimal grids.
\newblock {\em Numer. Math.}, 72(1):73--92, 1995.

\bibitem{krylov97}
N.~V. Krylov.
\newblock On the rate of convergence of finite-difference approximations for
  {B}ellman's equations.
\newblock {\em Algebra and Analysis, St.~Petersburg Math. J.}, 9(3):245--256,
  1997.

\bibitem{kushner1990numerical}
H.~J. Kushner.
\newblock Numerical methods for stochastic control problems in continuous time.
\newblock {\em SIAM J. Control Opt.}, 28(5):999--1048, 1990.

\bibitem{kushner2013numerical}
H.~J. Kushner and P.~G. Dupuis.
\newblock {\em Numerical methods for stochastic control problems in continuous
  time}, volume~24.
\newblock Springer, 2013.

\bibitem{ma2014unconditionally}
K.~Ma and P.~Forsyth.
\newblock An unconditionally monotone numerical scheme for the two factor
  uncertain volatility model.
\newblock {\em IMA Journal on Numerical Analysis}, 2016.
\newblock To appear.

\bibitem{monoyios2010utility}
M.~Monoyios.
\newblock Utility-based valuation and hedging of basis risk with partial
  information.
\newblock {\em Applied Mathematical Finance}, 17(6):519--551, 2010.

\bibitem{motzkinwasow53}
T.~S. Motzkin and W.~Wasow.
\newblock On the approximation of linear elliptic differential equations by
  difference equations with positive coefficients.
\newblock {\em J. Math. Phys.}, 31:253--259, 1953.

\bibitem{nocedalwright}
J.~Nocedal and S.~J. Wright.
\newblock {\em Numerical Optimization}.
\newblock Springer, 2nd edition, 2006.

\bibitem{sulem2006}
B.~{\O}ksendal and A.~Sulem.
\newblock {\em Applied Stochastic Control of Jump Diffusions}.
\newblock Springer, Berlin, 2nd edition, 2006.

\end{thebibliography}
\bibliographystyle{plain}

\end{document}